\documentclass[12pt]{article}
\usepackage{amssymb}
\usepackage{amsthm}
\usepackage{amsmath}
\usepackage{graphicx}
\usepackage{enumerate}

\DeclareMathOperator{\id}{id}

\DeclareMathOperator{\Con}{Con}

\newtheorem{theorem}{Theorem}[section]
\newtheorem{definition}[theorem]{Definition}
\newtheorem{lemma}[theorem]{Lemma}

\newtheorem{remark}[theorem]{Remark}
\newtheorem{example}[theorem]{Example}

\title{Extensions and congruences of partial lattices}
\author{Ivan~Chajda and Helmut~L\"anger}
\date{}
\begin{document}

\footnotetext{Support of the research of both authors by the Austrian Science Fund (FWF), project I~4579-N, and the Czech Science Foundation (GA\v CR), project 20-09869L, entitled ``The many facets of orthomodularity'', is gratefully acknowledged.}

\maketitle

\begin{abstract}
For a partial lattice $\mathbf L$ the so-called two-point extension is defined in order to extend $\mathbf L$ to a lattice. We are motivated by the fact that the one-point extension broadly used for partial algebras does not work in this case, i.e.\ the one-point extension of a partial lattice need not be a lattice. We describe these two-point extensions and prove several properties of them. We introduce the concept of a congruence on a partial lattice and show its relationship to the notion of a homomorphism and its connections with congruences on the corresponding two-point extension. In particular we prove that the quotient $\mathbf L/E$ of a partial lattice $\mathbf L$ by a congruence $E$ on $\mathbf L$ is again a partial lattice and that the two-point extension of $\mathbf L/E$ is isomorphic to the quotient lattice of the two-point extension $\mathbf L^*$ of $\mathbf L$ by the congruence on $\mathbf L^*$ generated by $E$. Several illustrative examples are enclosed.
\end{abstract}



\section{Introduction}

Although our paper is devoted to extensions of partial lattices to lattices as well as to congruences on partial lattices, we assume that it should be helpful to get a brief introduction to partial algebras in general and then apply these concepts to partial lattices. The source of the following concepts are the monographs \cite{B86} and \cite{B93} by P.~Burmeister and, concerning partial lattices, the paper \cite{CS} by the first author and Z.~Seidl. It is worth noticing that the closedness of classes of partial algebras with respect to quotients, subalgebras etc.\ was profoundly investigated by K.~Denecke in \cite D and by B.~Staruch and B.~Staruch in \cite{SS}.
    
In particular, we use the following concepts concerning various types of identities in partial algebras (see \cite{B93}).

Let $\tau=(n_i;i\in I)$ be a similarity type, $\mathbf A=(A,F)$ and $\mathbf B=(B,F)$ with $F=(f_i;i\in I)$ partial algebras of type $\tau$ and $p,q$ terms of type $\tau$. Then
\[
p\stackrel e=q
\]
means: $p$ and $q$ are defined and they are equal. We say that $\mathbf A$ satisfies the {\em weak identity}
\[
p\stackrel w\approx q
\]
if the following holds: If $a_1,\ldots,a_n\in A$ and $p(a_1,\ldots,a_n)$ and $q(a_1,\ldots,a_n)$ are defined then $p(a_1,\ldots,a_n)=q(a_1,\ldots,a_n)$. We say that $\mathbf A$ satisfies the {\em strong identity}
\[
p\stackrel s\approx q
\]
if the following holds: For $a_1,\ldots,a_n\in A$ the expression $p(a_1,\ldots,a_n)$ is defined if and only if $q(a_1,\ldots,a_n)$ is defined and in this case $p(a_1,\ldots,a_n)=q(a_1,\ldots,a_n)$. The {\em identity} $p\approx q$ is called {\em regular} if $p$ and $q$ contain the same variables.

The partial algebra $\mathbf A$ is called a {\em weak subalgebra} of $\mathbf B$ if $A\subseteq B$ and if for all $i\in I$ and all $a_1,\ldots,a_{n_i}\in A$, if $f_i(a_1,\ldots,a_{n_i})$ is defined in $\mathbf A$ then it is defined in $\mathbf B$ and has the same value. The weak subalgebra $\mathbf A$ of $\mathbf B$ is called a {\em subalgebra} of $\mathbf B$ if for all $i\in I$ and all $a_1,\ldots,a_{n_i}\in A$, if $f_i(a_1,\ldots,a_{n_i})$ is defined in $\mathbf B$ then it is defined in $\mathbf A$.

A {\em homomorphism} from $\mathbf A$ to $\mathbf B$ is a mapping $h\colon A\rightarrow B$ such that for all $i\in I$ and all $a_1,\ldots,a_{n_i},a\in A$, $f_i(a_1,\ldots,a_{n_i})\stackrel e=a$ implies $f_i\big(h(a_1),\ldots,h(a_{n_i})\big)\stackrel e=h(a)$. A {\em homomorphism} from $\mathbf A$ to $\mathbf B$ is called {\em closed} if for all $i\in I$ and all $a_1,\ldots,a_{n_i}\in A$, if $f_i\big(h(a_1),\ldots,h(a_{n_i})\big)$ is defined in $\mathbf B$ then $f_i(a_1,\ldots,a_{n_i})$ is defined in $\mathbf A$. It is easy to see that $\mathbf A$ is a weak subalgebra of $\mathbf B$ if and only if $A\subseteq B$ and $\id_A$ is a homomorphism from $\mathbf A$ to $\mathbf B$ and that $\mathbf A$ is a subalgebra of $\mathbf B$ if and only if $A\subseteq B$ and $\id_A$ is a closed homomorphism from $\mathbf A$ to $\mathbf B$. An {\em isomorphism} is a bijective closed homomorphism.

Let $\mathbf A=\big(A,(f_i;i\in I)\big)$ be a partial algebra and assume $c\notin A$. Then the {\em one-point extension} $\bar{\mathbf A}$ of $\mathbf A$ is defined as follows:
\begin{enumerate}[{\rm(i)}]
\item If all $f_i$ are defined everywhere then $\bar{\mathbf A}:=\mathbf A$.
\item if at least one $f_i$ is not defined everywhere then $\bar{\mathbf A}:=\big(A\cup\{c\},(g_i;i\in I)\big)$ where for all $i\in I$ and all $a_1,\ldots,a_{n_i}\in A\cup\{c\}$
\[
g_i(a_1,\ldots,a_{n_i}):=\left\{
\begin{array}{ll}
f_i(a_1,\ldots,a_{n_i}) & \text{if }f_i(a_1,\ldots,a_{n_i})\text{ is defined in }\mathbf A, \\
c                       & \text{otherwise}.
\end{array}
\right.
\]
\end{enumerate}
It was proved in \cite D and \cite{SS} that a class of partial algebras can be described by a set of strong and regular identities if and only if it is closed under the formation of subalgebras, closed homomorphic images, direct products, initial segments and one-point extensions.

However, not every interesting class of partial algebras can be described by strong and regular identities. For example, partial lattices treated in \cite{CS} are described by identities containing those for absorption that are neither regular nor strong. This is in accordance with the fact that the one-point extension of a partial lattice is not a lattice in general. We want to present another construction for partial lattices, the so-called {\em two-point extension}, which preserves also weak identities like the absorption identities.

\section{Partial lattices}

There exist several definitions of a partial lattice. For our purposes we adopt the following one from \cite{CS}.

\begin{definition}\label{def2}
A {\em partial lattice} is a partial algebra $(L,\vee,\wedge)$ of type $(2,2)$ satisfying the following identities:
\begin{align*}
                     x\vee x\stackrel s\approx x & \text{ and }x\wedge x\stackrel s\approx x, \\
               x\vee y\stackrel s\approx y\vee x & \text{ and }x\wedge y\stackrel s\approx y\wedge x, \\
(x\vee y)\vee z\stackrel s\approx x\vee(y\vee z) & \text{ and }(x\wedge y)\wedge z\stackrel s\approx x\wedge(y\wedge z)
\end{align*}
as well as the following {\em duality conditions}:
\begin{align*}
  x\vee y\stackrel e=x & \text{ implies }x\wedge y\stackrel e=y, \\
x\wedge y\stackrel e=x & \text{ implies }x\vee y\stackrel e=y.
\end{align*}
We call weak subalgebras of a partial lattice also {\em partial sublattices}.
\end{definition}

The following lemma was proved in \cite{CS}. For the sake of completeness we repeat the proof.

\begin{lemma}\label{lem2}
Let $\mathbf L$ be a partial lattice. Then the following hold:
\begin{enumerate}[{\rm(i)}]
\item $\mathbf L$ satisfies $(x\vee y)\wedge x\stackrel w\approx x$ and $(x\wedge y)\vee x\stackrel w\approx x$.
\item If $\mathbf L$ satisfies $(x\vee y)\wedge x\stackrel s\approx x$ and $(x\wedge y)\vee x\stackrel s\approx x$ then $\mathbf L$ is a lattice.
\end{enumerate}
\end{lemma}

\begin{proof}
\
\begin{enumerate}[(i)]
\item Let $a,b\in L$ and assume $(a\vee b)\wedge a$ to be defined. Then $a\vee b$ is defined. Since $a\vee a\stackrel e=a$ we conclude that $(a\vee a)\vee b$ is defined. Hence $a\vee(a\vee b)$ and $(a\vee b)\vee a$ are defined and we obtain
\[
(a\vee b)\vee a\stackrel e=a\vee(a\vee b)\stackrel e=(a\vee a)\vee b\stackrel e=a\vee b.
\]
Applying the duality conditions yields $(a\vee b)\wedge a\stackrel e=a$. The second weak identity follows analogously.
\item is clear.
\end{enumerate}
\end{proof}

Let $\mathbf P=(P,\leq)$ be a poset and $a,b\in P$. We define
\begin{align*}
L(a,b) & :=\{x\in P\mid x\leq a,b\}, \\
U(a,b) & :=\{x\in P\mid a,b\leq x\}.
\end{align*}
The following notions as well as the following result are taken from \cite{CS}:

The poset $\mathbf P$ is said to satisfy the
\begin{itemize}
\item {\em lower bound property} (LBP) if for all $x,y\in L$ the set $L(x,y)$ is either empty or possesses a greatest element,
\item {\em upper bound property} (UBP) if for all $x,y\in L$ the set $U(x,y)$ is either empty or possesses a smallest element.
\end{itemize}
Let $\mathbf L=(L,\vee,\wedge)$ be a partial lattice. On $L$ we define the binary relations $\leq_\vee$ and $\leq_\wedge$ as follows:
\begin{align*}
  x\leq_\vee y & \text{ if and only if }x\vee y\stackrel e=y, \\
x\leq_\wedge y & \text{ if and only if }x\wedge y\stackrel e=x.
\end{align*}
Due to the duality conditions, the relations $\leq_\vee$ and $\leq_\wedge$ coincide and form a partial order relation on $L$. We call it the {\em induced order} of $\mathbf L$ and denote them by $\leq$.

We call a poset satisfying the {\rm(LBP)} and the {\rm(UBP)} a {\em partially lattice-ordered set}.

\begin{example}
The poset depicted in Figure~1

\vspace*{-3mm}

\begin{center}
\setlength{\unitlength}{7mm}
\begin{picture}(6,6)
\put(2,1){\circle*{.3}}
\put(1,2){\circle*{.3}}
\put(3,2){\circle*{.3}}
\put(1,4){\circle*{.3}}
\put(3,4){\circle*{.3}}
\put(2,5){\circle*{.3}}
\put(2,1){\line(-1,1)1}
\put(2,1){\line(1,1)1}
\put(1,2){\line(0,1)2}
\put(1,2){\line(1,1)2}
\put(3,2){\line(-1,1)2}
\put(3,2){\line(0,1)2}
\put(2,5){\line(-1,-1)1}
\put(2,5){\line(1,-1)1}
\put(1.85,.25){$0$}
\put(.3,1.85){$a$}
\put(.3,3.85){$c$}
\put(3.4,1.85){$b$}
\put(3.4,3.85){$d$}
\put(1.85,5.4){$1$}
\put(1.25,-.75){{\rm Fig.~1}}
\end{picture}
\end{center}

\vspace*{3mm}

is not a partially lattice-ordered set since e.g.\ $U(a,b)=\{c,d,1\}\neq\emptyset$, but $U(a,b)$ has no smallest element.
\end{example}

The following theorem was partly proved in \cite{CS}. For the sake of completeness we provide a complete proof.

\begin{theorem}
\
\begin{enumerate}[{\rm(i)}]
\item Let $\mathbf L=(L,\vee,\wedge)$ be a partial lattice and $\leq$ its induced order. Then $\mathbb P(\mathbf L):=(L,\leq)$ is a partially lattice-ordered set.
\item Let $\mathbf P=(P,\leq)$ be a partially lattice-ordered set and define partial binary operations $\vee$ and $\wedge$ on $P$ as follows:
\begin{align*}
  x\vee y & \left\{
\begin{array}{ll}
:=\sup(x,y)         & \text{if }U(x,y)\neq\emptyset, \\
\text{is undefined} & \text{otherwise},
\end{array}
\right. \\
x\wedge y & \left\{
\begin{array}{ll}
:=\inf(x,y)         & \text{if }L(x,y)\neq\emptyset, \\
\text{is undefined} & \text{otherwise}
\end{array}
\right.
\end{align*}
{\rm(}$x,y\in P${\rm)}. Then $\mathbb L(\mathbf P):=(P,\vee,\wedge)$ is a partial lattice.
\item Let $\mathbf L=(L,\vee,\wedge)$ be a partial lattice. Then $\mathbb L\big(\mathbb P(\mathbf L)\big)=\mathbf L$.
\item Let $\mathbf P=(P,\leq)$ be a partially lattice-ordered set. Then $\mathbb P\big(\mathbb L(\mathbf P)\big)=\mathbf P$.
\end{enumerate}
\end{theorem}

\begin{proof}
\
\begin{enumerate}[(i)]
\item Let $a,b,c\in L$. Then $a\vee a\stackrel e=a$ and hence $a\leq a$. If $a\leq b\leq a$ then $a\vee b\stackrel e=b$ and $b\vee a\stackrel e=a$ and hence $a\stackrel e=b\vee a\stackrel e=a\vee b\stackrel e=b$. If $a\leq b\leq c$ then
\[
c\stackrel e=b\vee c\stackrel e=(a\vee b)\vee c\stackrel e=a\vee(b\vee c)\stackrel e=a\vee c,
\]
i.e.\ $a\leq c$. Now assume $c\in U(a,b)$. Then $a\vee c\stackrel e=c$ and $b\vee c\stackrel e=c$ and hence
\[
c\stackrel e=a\vee c\stackrel e=a\vee(b\vee c)\stackrel e=(a\vee b)\vee c,
\]
i.e.\ $a\vee b$ is defined and $a\vee b\leq c$. Because of
\begin{align*}
a\vee(a\vee b) & \stackrel e=(a\vee a)\vee b\stackrel e=a\vee b, \\
b\vee(a\vee b) & \stackrel e=(a\vee b)\vee b\stackrel e=a\vee(b\vee b)\stackrel e=a\vee b
\end{align*}
we have $a,b\leq a\vee b$. Together we obtain that in case $U(a,b)\neq\emptyset$ the element $a\vee b$ is defined and coincides with $\sup(a,b)$. The dual assertion can be proved analogously.
\item Let $a,b,c\in P$. Clearly, $a\vee a\stackrel s=a$, $a\wedge a\stackrel s=a$, $a\vee b\stackrel s=b\vee a$ and $a\wedge b\stackrel s=b\wedge a$. Assume $(a\vee b)\vee c$ to be defined. Since $(a\vee b)\vee c\in U(b,c)$ the element $b\vee c$ is defined and since $(a\vee b)\vee c\in U(a,b\vee c)$ the element $a\vee(b\vee c)$ is defined and we obtain $(a\vee b)\vee c=\sup(a,b,c)=a\vee(b\vee c)$. If, conversely, $a\vee(b\vee c)$ is defined then
\[
a\vee(b\vee c)\stackrel e=(b\vee c)\vee a\stackrel e=(c\vee b)\vee a\stackrel e=c\vee(b\vee a)\stackrel e=(b\vee a)\vee c\stackrel e=(a\vee b)\vee c.
\]
The strong identity $(x\wedge y)\wedge z\stackrel s\approx x\wedge(y\wedge z)$ can be proved analogously. Finally, the duality conditions can be easily verified.
\item Let $\mathbb P(\mathbf L)=(L,\leq)$ and $\mathbb L\big(\mathbb P(\mathbf L)\big)=(L,\sqcup,\sqcap)$ and $a,b\in L$. According to the proof of (i), if $U(a,b)\neq\emptyset$ then $a\vee b$ is defined and $a\vee b=\sup(a,b)$. Conversely, if $a\vee b$ is defined then
\begin{align*}
a\vee(a\vee b) & \stackrel e=(a\vee a)\vee b\stackrel e=a\vee b, \\
b\vee(a\vee b) & \stackrel e=(a\vee b)\vee b\stackrel e=a\vee(b\vee b)\stackrel e=a\vee b
\end{align*}
and hence $a\vee b\in U(a,b)$ showing $U(a,b)\neq\emptyset$. Hence the following are equivalent: $a\vee b$ is defined; $U(a,b)\neq\emptyset$; $a\sqcup b$ is defined. In this case we have $a\sqcup b=\sup(a,b)=a\vee b$. Analogously, one can prove that $a\sqcap b$ is defined if and only if $a\wedge b$ is defined and in this case $a\sqcap b=a\wedge b$.
\item If $\mathbb L(\mathbf P)=(P,\vee,\wedge)$, $\mathbb P\big(\mathbb L(\mathbf P)\big)=(P,\sqsubseteq)$ and $a,b\in P$ then the following are equivalent: $a\sqsubseteq b$; $a\vee b\stackrel e=b$; $\sup(a,b)\stackrel e=b$; $a\leq b$. This shows $\mathbb P\big(\mathbb L(\mathbf P)\big)=\mathbf P$.
\end{enumerate}
\end{proof}

\section{Extensions}

As mentioned above, the one-point extension of a partial lattice need not be a lattice. In order to avoid this difficulty, we introduce the two-point extension of a partial lattice as follows.

\begin{definition}
Let $\mathbf L=(L,\vee,\wedge)$ be a partial lattice and $\leq$ its induced order and assume $0,1\notin L$. Then the {\em two-point extension} $\mathbf L^*=(L^*,\leq^*)$ of $\mathbf L$ is defined as follows:
\begin{enumerate}[{\rm(i)}]
\item If $\vee$ and $\wedge$ are defined everywhere then
\[
\mathbf L^*:=(L,\leq),
\]
\item if $\wedge$ is defined everywhere, but $\vee$ is not then
\[
\mathbf L^*:=(L\cup\{1\},\leq\cup(L^*\times\{1\}),
\]
\item if $\vee$ is defined everywhere, but $\wedge$ is not then
\[
\mathbf L^*:=(L\cup\{0\},\leq\cup(\{0\}\times L^*),
\]
\item if neither $\vee$ nor $\wedge$ is defined everywhere then
\[
\mathbf L^*:=(L\cup\{0,1\},\leq\cup(\{0\}\times L^*)\cup(L^*\times\{1\}).
\]
\end{enumerate}
\end{definition}

Clearly, $\mathbf L$ is a partial sublattice of $\mathbf L^*$. If $\mathbf L$ is an infinite lattice having neither a smallest nor a greatest element then its two-point extension coincides with $\mathbf L$. Hence, one should have in mind that $\mathbf L^*$ neither need have a smallest nor a greatest element.

The following lemma is obvious.

\begin{lemma}
Let $\mathbf L=(L,\vee,\wedge)$ be a partial lattice and $a,b,c,d\in L$. Then
\begin{enumerate}[{\rm(i)}]
\item $\mathbf L^*$ is a lattice with lattice operations
\[
x\vee^* y:=\sup_{\leq^*}(x,y)\text{ and }x\wedge^* y:=\inf_{\leq^*}(x,y)
\]
{\rm(}$x,y\in L^*${\rm)},
\item $a\vee b\stackrel e=c$ if and only if $a\vee^*b=c$, and $a\wedge b\stackrel e=d$ if and only if $a\wedge^*b=d$,
\item $a\vee^*b=\left\{
\begin{array}{ll}
a\vee b & \text{ if }U(a,b)\neq\emptyset, \\
1       & \text{ otherwise},
\end{array}
\right.$ and $a\wedge^*b=\left\{
\begin{array}{ll}
a\wedge b & \text{ if }L(a,b)\neq\emptyset, \\
0         & \text{ otherwise}.
\end{array}
\right.$
\end{enumerate}
\end{lemma}

For homomorphisms of partial lattices and their two-point extensions, we can prove the following result.

\begin{theorem}\label{th2}
Let $\mathbf L_i=(L_i,\vee,\wedge),i=1,2,$ be partial lattices.
\begin{enumerate}[{\rm(i)}]
\item Let $h^*$ be a homomorphism from $\mathbf L_1^*$ to $\mathbf L_2^*$ with $h^*(L_1)\subseteq L_2$. Then $h^*|L_1$ is a homomorphism from $\mathbf L_1$ to $\mathbf L_2$.
\item Let $h$ be a closed homomorphism from $\mathbf L_1$ to $\mathbf L_2$. Then there exists some homomorphism $h^*$ from $\mathbf L_1^*$ to $\mathbf L_2^*$ with $h^*|L_1=h$.
\end{enumerate}
\end{theorem}

\begin{proof}
\
\begin{enumerate}[(i)]
\item Put $h:=h^*|L_1$ and let $a,b\in L_1$. First assume $a\vee b$ to be defined. If $h(a)\vee h(b)$ would not be defined then we would conclude
\[
1_{\mathbf L_2}=h(a)\vee^*h(b)=h^*(a)\vee^*h^*(b)=h^*(a\vee^*b)=h(a\vee b)\in h(L_1)\subseteq L_2,
\]
a contradiction. Hence $h(a)\vee h(b)$ is defined and
\[
h(a\vee b)=h^*(a\vee^*b)=h^*(a)\vee^*h^*(b)=h(a)\vee h(b).
\]
Analogously, one can prove that whenever $a\wedge b$ is defined, also $h(a)\wedge h(b)$ is defined and $h(a\wedge b)=h(a)\wedge h(b)$.
\item Define $h^*\colon L_1^*\rightarrow L_2^*$ by
\begin{align*}
                h^*(x) & :=h(x)\text{ for all }x\in L_1, \\
h^*(0_{\mathbf L_1^*}) & :=0_{\mathbf L_2^*}\text{ if }0_{\mathbf L_1^*}\in L_1^*, \\
h^*(1_{\mathbf L_1^*}) & :=1_{\mathbf L_2^*}\text{ if }1_{\mathbf L_1^*}\in L_1^*.
\end{align*}
$\bullet$ $h^*$ is well-defined. \\
If $0_{\mathbf L_2^*}\notin L_2^*$ then $h(x)\wedge h(y)$ is defined for all $x,y\in L_1$ and hence, since $h$ is closed, $x\wedge y$ is defined for all $x,y\in L_1$ showing $0_{\mathbf L_1^*}\notin L_1^*$. Analogously, one can prove that $1_{\mathbf L_2^*}\notin L_2^*$ implies $1_{\mathbf L_1^*}\notin L_1^*$. \\
$\bullet$ $h^*$ is a homomorphism from $\mathbf L_1^*$ to $\mathbf L_2^*$. \\
Let $a,b\in L_1$. If $a\vee b$ is defined then $h(a)\vee h(b)$ is defined and
\[
h^*(a\vee^*b)=h(a\vee b)=h(a)\vee h(b)=h^*(a)\vee^*h(b).
\]
If $a\vee b$ is not defined then, since $h$ is closed, $h(a)\vee h(b)$ is not defined, too, and
\[
h^*(a\vee^*b)=h^*(1_{\mathbf L_1^*})=1_{\mathbf L_2^*}=h(a)\vee^*h(b)=h^*(a)\vee^*h^*(b).
\]
Analogously, one can show $h^*(a\wedge^*b)=h(a)\wedge^*h(b)$. It is not hard to prove
\begin{align*}
  h^*(x\vee^*y) & =h^*(x)\vee^*h(y), \\
h^*(x\wedge^*y) & =h^*(x)\wedge^*h(y)
\end{align*}
for all $(x,y)\in(L_1^*)^2\setminus L_1^2$. \\
$\bullet$ $h^*|L_1=h$. \\
This is clear.
\end{enumerate}
\end{proof}

The following example shows that $\mathbf L^*$ need not be the smallest lattice including $\mathbf L$ as a partial sublattice. Moreover, we show that in general the operator $^*$ does not preserve inclusion.

\begin{example}
The partial lattice $\mathbf L_1=(L_1,\vee,\wedge)$ visualized in Figure~2:

\vspace*{-1mm}

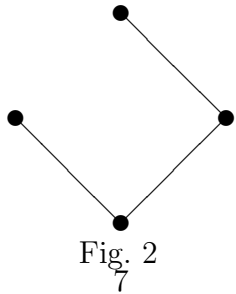
\begin{center}
\setlength{\unitlength}{7mm}
\begin{picture}(4,4)
\put(0,2){\circle*{.3}}
\put(2,0){\circle*{.3}}
\put(2,4){\circle*{.3}}
\put(4,2){\circle*{.3}}
\put(2,0){\line(-1,1)2}
\put(2,0){\line(1,1)2}
\put(2,4){\line(1,-1)2}
\put(1.2,-.75){{\rm Fig.~2}}
\end{picture}
\end{center}

\vspace*{5mm}

is a partial sublattice of the partial lattice $\mathbf L_2=(L_2,\vee,\wedge)$ depicted in Figure~3:

\vspace*{-1mm}

\begin{center}
\setlength{\unitlength}{7mm}
\begin{picture}(4,4)
\put(0,2){\circle*{.3}}
\put(2,0){\circle*{.3}}
\put(2,4){\circle*{.3}}
\put(4,2){\circle*{.3}}
\put(2,0){\line(-1,1)2}
\put(2,0){\line(1,1)2}
\put(2,4){\line(-1,-1)2}
\put(2,4){\line(1,-1)2}
\put(1.2,-.75){{\rm Fig.~3}}
\end{picture}
\end{center}

\vspace*{3mm}

and one can see that $\id_{L_1}$ is a homomorphism from $\mathbf L_1$ to $\mathbf L_2$ which is not closed. Although $L_1\subseteq L_2$, we do not have $L_1^*\subseteq L_2^*$ since $\mathbf L_1^*\cong\mathbf N_5$ whereas $\mathbf L_2^*=\mathbf L_2$. Moreover, $\mathbf L_1^*$ is not the smallest lattice including $\mathbf L_1$ as a partial sublattice.
\end{example}

\section{Congruences and quotient partial lattices}

We are going to introduce the concept of congruence on a partial lattice and show how to define the partial lattice operations for congruence classes in a natural way.

\begin{definition}\label{def1}
Let $\mathbf L=(L,\vee,\wedge)$ be a partial lattice. By a {\em congruence} on $\mathbf L$ is meant an equivalence relation $E$ on $L$ satisfying $\Theta(E)\cap L^2=E$ where $\Theta(E)$ denotes the congruence on the lattice $\mathbf L^*$ generated by $E$. Let $\Con\mathbf L$ denote the set of all congruences on $\mathbf L$. For $E\in\Con\mathbf L$ we define two partial operation $\vee$ and $\wedge$ on $L/E$ by
\begin{align*}
  [x]E\vee[y]E & \left\{
\begin{array}{ll}
:=[x\vee^*y]\big(\Theta(E)\big)\cap L & \text{if this set is non-empty}, \\
\text{is undefined}                   & \text{otherwise},
\end{array}
\right. \\
[x]E\wedge[y]E & \left\{
\begin{array}{ll}
:=[x\wedge^*y]\big(\Theta(E)\big)\cap L & \text{if this set is non-empty}, \\
\text{is undefined}                     & \text{otherwise}.
\end{array}
\right.
\end{align*}
\end{definition}

It is easy to see that the operations $\vee$ and $\wedge$ on $L/E$ are well-defined.

The next lemma shows that congruences as defined above have the expected properties.

\begin{lemma}
Let $\mathbf L=(L,\vee,\wedge)$ be a partial lattice. Then
\begin{enumerate}[{\rm(i)}]
\item $\Con\mathbf L=\{\Theta\cap L^2\mid\Theta\in\Con\mathbf L^*\}$,
\item $(\Con\mathbf L,\subseteq)$ is a complete lattice.
\end{enumerate}
\end{lemma}

\begin{proof}
\
\begin{enumerate}[(i)]
\item Let $E$ be an equivalence relation on $L$. If $E\in\Con\mathbf L$ then $\Theta(E)\in\Con\mathbf L^*$ and $E=\Theta(E)\cap L^2$. If, conversely, $\Theta\in\Con\mathbf L^*$ and $\Theta\cap L^2=E$ then $E\subseteq\Theta$ and hence $\Theta(E)\subseteq\Theta$ whence
\[
E\subseteq\Theta(E)\cap L^2\subseteq\Theta\cap L^2=E
\]
which implies $\Theta(E)\cap L^2=E$ and hence $E\in\Con\mathbf L$.
\item If $E_i\in\Con\mathbf L$ for all $i\in I$ then for every $i\in I$ there exists some $\Theta_i\in\Con\mathbf L^*$ satisfying $\Theta_i\cap L^2=E_i$ and hence $\bigcap\limits_{i\in I}\Theta_i\in\Con\mathbf L^*$ and
\[
\bigcap_{i\in I}E_i=\bigcap_{i\in I}(\Theta_i\cap L^2)=(\bigcap_{i\in I}\Theta_i)\cap L^2\in\Con\mathbf L
\]
according to (i).
\end{enumerate}
\end{proof}

The following lemma describes the partial lattice operations in quotients of partial lattices.

\begin{lemma}\label{lem1}
Let $\mathbf L=(L,\vee,\wedge)$ be a partial lattice, $a,b\in L$ and $E\in\Con\mathbf L$. Then
\[
[a]E\vee[b]E\left\{
\begin{array}{ll}
=[a\vee b]E         & \text{if }U(a,b)\neq\emptyset, \\
\text{is undefined} & \text{if }U(a,b)=\emptyset\text{ and }[1]\big(\Theta(E)\big)=\{1\}, \\
=[\alpha]E          & \text{if }U(a,b)=\emptyset\text{ and }[1]\big(\Theta(E)\big)\neq\{1\}
\end{array}
\right.
\]
where $\alpha\in[1]\big(\Theta(E)\big)\setminus\{1\}$.
\end{lemma}

\begin{proof}
If $U(a,b)\neq\emptyset$ then
\[
[a\vee^*b]\big(\Theta(E)\big)\cap L=[a\vee b]\big(\Theta(E)\big)\cap L=[a\vee b]E\neq\emptyset
\]
and hence $[a]E\vee[b]E=[a\vee b]E$. If $U(a,b)=\emptyset$ and $[1]\big(\Theta(E)\big)=\{1\}$ then
\[
[a\vee^*b]\big(\Theta(E)\big)\cap L=[1]\big(\Theta(E)\big)\cap L=\{1\}\cap L=\emptyset
\]
and hence $[a]E\vee[b]E$ is not defined. Assume, finally, $U(a,b)=\emptyset$ and $[1]\big(\Theta(E)\big)\neq\{1\}$. Then $[1]\big(\Theta(E)\big)\cap L=\emptyset$ would imply $[1]\big(\Theta(E)\big)=\{0,1\}$ and hence $\Theta(E)=(L^*)^2$ whence $L=L^*\cap L=[1]\big(\Theta(E)\big)\cap L=\emptyset$, a contradiction. This shows that there exists some $\alpha\in[1]\big(\Theta(E)\big)\cap L$. Now
\[
[a\vee^*b]\big(\Theta(E)\big)\cap L=[1]\big(\Theta(E)\big)\cap L=[\alpha]\big(\Theta(E)\big)\cap L=[\alpha]E\neq\emptyset
\]
and hence $[a]E\vee[b]E=[\alpha]E$.
\end{proof}

\begin{remark}
Due to duality, the dual version of Lemma~\ref{lem1} holds, too.
\end{remark}

The main point is to show that the quotient partial algebra $(L/E,\vee,\wedge)$ as defined above is again a partial lattice, the so-called {\em quotient partial lattice by the congruence $E$}.

\begin{theorem}
Let $\mathbf L=(L,\vee,\wedge)$ be a partial lattice and $E\in\Con\mathbf L$. Then $\mathbf L/E:=(L/E,\vee,\wedge)$ is a partial lattice.
\end{theorem}

\begin{proof}
In this proof we often use Lemma~\ref{lem1}. We prove only a part of the conditions mentioned in Definition~\ref{def2}, the remaining conditions follow by duality. Let $a,b,c\in L$. It is evident directly by Definition~\ref{def1} that the partial operations $\vee$ and $\wedge$ on $L/E$ satisfy strong idempotency and strong commutativity. We need to prove also strong associativity. Assume $([a]E\vee[b]E)\vee[c]E$ to be defined. We have
\[
U(a,b),U(a\vee b,c)\neq\emptyset\text{ if and only if }U(b,c),U(a,b\vee c)\neq\emptyset
\]
and in this case we conclude
\[
(a\vee b)\vee c=\sup(a,b,c)=a\vee(b\vee c).
\]
$\bullet$ Case~1. $U(a,b),U(a\vee b,c)\neq\emptyset$. \\
Then $[a]E\vee([b]E\vee[c]E)$ is defined and
\begin{align*}
([a]E\vee[b]E)\vee[c]E & =[a\vee b]E\vee[c]E=[(a\vee b)\vee c]E=[a\vee(b\vee c)]E= \\
                       & =[a]E\vee[b\vee c]E=[a]E\vee([b]E\vee[c]E).
\end{align*}
$\bullet$ Case 2. $U(a,b)=\emptyset$ or ($U(a,b)\neq\emptyset$ and $U(a\vee b,c)=\emptyset$). \\
Then $1\in L^*$, $[1]\big(\Theta(E)\big)\neq\{1\}$ and $[a]E\vee([b]E\vee[c]E)$ is defined. According to the proof of Lemma~\ref{lem1} there exists some $\alpha\in[1]\big(\Theta(E)\big)\cap L$ and
\begin{align*}
[x]E\vee[\alpha]E & =[x\vee^*\alpha]\big(\Theta(E)\big)\cap L=[x\vee^*1]\big(\Theta(E)\big)\cap L=[1]\big(\Theta(E)\big)\cap L= \\
                  & =[\alpha]\big(\Theta(E)\big)\cap L=[\alpha]E
\end{align*}
for all $x\in L$. \\
$\bullet$ Case~2a. $U(a,b)=U(b,c)=\emptyset$. \\
Then
\[
([a]E\vee[b]E)\vee[c]E=[\alpha]E\vee[c]E=[\alpha]E=[a]E\vee[\alpha]E=[a]E\vee([b]E\vee[c]E).
\]
Hence $([a]E\vee[b]E)\vee[c]E=[a]E\vee([b]E\vee[c]E)$. \\
$\bullet$ Case~2b. $U(a,b)=\emptyset$ and $U(b,c)\neq\emptyset$. \\
Then $U(a,b\vee c)=\emptyset$ and
\[
([a]E\vee[b]E)\vee[c]E=[\alpha]E\vee[c]E=[\alpha]E=[a]E\vee[b\vee c]E=[a]E\vee([b]E\vee[c]E).
\]
$\bullet$ Case~2c. $U(a,b)\neq\emptyset$ and $U(b,c)=\emptyset$. \\
Then $U(a\vee b,c)=\emptyset$ and
\[
([a]E\vee[b]E)\vee[c]E=[a\vee b]E\vee[c]E=[\alpha]E=[a]E\vee[\alpha]E=[a]E\vee([b]E\vee[c]E).
\]
$\bullet$ Case~2d. $U(a,b),U(b,c)\neq\emptyset$ and $U(a\vee b,c)=U(a,b\vee c)=\emptyset$. \\
Then
\[
([a]E\vee[b]E)\vee[c]E=[a\vee b]E\vee[c]E=[\alpha]E=[a]E\vee[b\vee c]E=[a]E\vee([b]E\vee[c]E).
\]
If, conversely $[a]E\vee([b]E\vee[c]E)$ is defined then
\begin{align*}
[a]E\vee([b]E\vee[c]E) & \stackrel e=([b]E\vee[c]E)\vee[a]E\stackrel e=([c]E\vee[b]E)\vee[a]E\stackrel e= \\
                       & \stackrel e=[c]E\vee([b]E\vee[a]E)\stackrel e=([b]E\vee[a]E)\vee[c]E\stackrel e=([a]E\vee[b]E)\vee[c]E.
\end{align*}
Finally, the duality conditions remain to be proved. If $[a]E\vee[b]E\stackrel e=[a]E$ then $a\in[a]E\subseteq[a\vee^*b]\big(\Theta(E)\big)$, i.e.\ $a\mathrel{\Theta(E)}(a\vee^*b)$ and hence $(a\wedge^*b)\mathrel{\Theta(E)}\big((a\vee^*b)\wedge^*b\big)=b$ whence $[a\wedge^*b]\big(\Theta(E)\big)\cap L=[b]\big(\Theta(E)\big)\cap L=[b]E\neq\emptyset$, i.e.\ $[a]E\wedge[b]E\stackrel e=[b]E$.
\end{proof}

Similar to the situation for algebras there exist natural relationships between congruences and homomorphisms also for partial lattices.

\begin{lemma}
Let $\mathbf L=(L,\vee,\wedge)$ be a partial lattice and $E\in\Con\mathbf L$ and define $h\colon L\rightarrow L/E$ by $h(x):=[x]E$ for all $x\in L$. Then $h$ is a homomorphism from $\mathbf L$ to $\mathbf L/E$ and $\ker h=E$. If, in addition, $[0_{\mathbf L^*}]\big(\Theta(E)\big)=\{0_{\mathbf L^*}\}$ provided $0_{\mathbf L^*}\in L^*$ and $[1_{\mathbf L^*}]\big(\Theta(E)\big)=\{1_{\mathbf L^*}\}$ provided $1_{\mathbf L^*}\in L^*$ then $h$ is a closed homomorphism from $\mathbf L$ to $\mathbf L/E$.
\end{lemma}

\begin{proof}
This easily follows from Lemma~\ref{lem1}.
\end{proof}

For partial lattices we can now prove the following version of the Homomorphism Theorem.

\begin{theorem}
For $i=1,2$ let $\mathbf L_i=(L_i,\vee,\wedge)$ be partial lattices and $h$ a closed homomorphism from $\mathbf L_1$ to $\mathbf L_2$. Then $\ker h\in\Con\mathbf L_1$. If, in addition, $[0_{\mathbf L_1^*}]\big(\Theta(\ker h)\big)=\{0_{\mathbf L_1^*}\}$ provided $0_{\mathbf L_1^*}\in L_1^*$ and $[1_{\mathbf L_1^*}]\big(\Theta(\ker h)\big)=\{1_{\mathbf L_1^*}\}$ provided $1_{\mathbf L_1^*}\in L_1^*$ then $\big(h(L_1),\vee,\wedge\big)$ {\rm(}with partial operations defined as in $\mathbf L_2${\rm)} is a partial lattice which is isomorphic to $\mathbf L_1/(\ker h)$.
\end{theorem}

\begin{proof}
According to Theorem~\ref{th2} there exists some homomorphism $h^*$ from $\mathbf L_1^*$ to $\mathbf L_2^*$ satisfying $h^*|L_1=h$. Now $\ker h^*\in\Con\mathbf L_1$ and  hence $\ker h=\ker h^*\cap L_1^2\in\Con\mathbf L_1$. If the additional condition holds then, according to Lemma~\ref{lem1}, the mapping $h(x)\mapsto[x]\ker h$ is a well-defined isomorphism from $\big(h(L_1),\vee,\wedge\big)$ to $\mathbf L_1/(\ker h)$ since for $a,b\in L_1$ the following are equivalent: $h(a)\vee h(b)$ exists in $\mathbf L_2$; $a\vee b$ exists in $\mathbf L_1$; $[a]\ker h\vee[b]\ker h$ exists in $\mathbf L_1/(\ker h)$. Analogous statements hold for $\wedge$ instead of $\vee$.
\end{proof}

Now we show that our concept of a quotient partial lattice $\mathbf L/E$ is sound, i.e.\ its two-point extension is isomorphic to the quotient lattice of the two-point extension of $\mathbf L$ with respect to $\Theta(E)$.

\begin{theorem}\label{th1}
Let $\mathbf L=(L,\vee,\wedge)$ be a partial lattice and $E\in\Con\mathbf L$. Then $(\mathbf L/E)^*\cong\mathbf L^*/\big(\Theta(E)\big)$.
\end{theorem}

\begin{proof}
Define $f\colon(L/E)^*\rightarrow L^*/\big(\Theta(E)\big)$ in the following way:
\begin{align*}
               f([x]E) & :=[x]\big(\Theta(E)\big)\text{ for all }x\in L, \\
f(0_{(\mathbf L/E)^*}) & :=[0_{\mathbf L^*}]\big(\Theta(E)\big)\text{ if }0_{(\mathbf L/E)^*}\in(L/E)^*, \\
f(1_{(\mathbf L/E)^*}) & :=[1_{\mathbf L^*}]\big(\Theta(E)\big)\text{ if }1_{(\mathbf L/E)^*}\in(L/E)^*.
\end{align*}
We will prove that $f$ is an isomorphism from $(\mathbf L/E)^*$ to $\mathbf L^*/\big(\Theta(E)\big)$. Let $a,b\in L$. \\
$\bullet$ $f$ is well-defined. \\
If $[a]E=[b]E$ then $[a]\big(\Theta(E)\big)=[b]\big(\Theta(E)\big)$. We prove that $U([a]E,[b]E)=\emptyset$ implies $U(a,b)=\emptyset$. Assume $U(a,b)\neq\emptyset$, say $c\in U(a,b)$. Then $a\vee c\stackrel e=c$ and $b\vee c\stackrel e=c$. Hence
\[
[a]E\vee[c]E=[a\vee^*c]\big(\Theta(E)\big)\cap L=[c]\big(\Theta(E)\big)\cap L=[c]E
\]
and, analogously, $[b]E\vee[c]E=[c]E$, i.e.\ $[c]E\in U([a]E,[b]E)\neq\emptyset$. Hence, if $1_{(\mathbf L/E)^*}\in(L/E)^*$ then there exist $d,e\in L$ with $U([d]E,[e]E)=\emptyset$. Therefore $U(d,e)=\emptyset$ which shows $1_{\mathbf L^*}\in L^*$. Analogously, one can prove that $0_{(\mathbf L/E)^*}\in(L/E)^*$ implies $0_{\mathbf L^*}\in L^*$. \\
$\bullet$ $f$ is injective. \\
If $[a]\big(\Theta(E)\big)=[b]\big(\Theta(E)\big)$ then $(a,b)\in\Theta(E)\cap L^2=E$, i.e.\ $[a]E=[b]E$. Now assume $1_{(\mathbf L/E)^*}\in(L/E)^*$ and $[a]\big(\Theta(E)\big)=[1_{\mathbf L^*}]\big(\Theta(E)\big)$. Then there exist $c,d\in L$ with $U([c]E,[d]E)=\emptyset$. Now we have
\begin{align*}
[c]E\vee[a]E & =[c\vee^*a]\big(\Theta(E)\big)\cap L=[c\vee^*1_{\mathbf L^*}]\big(\Theta(E)\big)\cap L=[1_{\mathbf L^*}]\big(\Theta(E)\big)\cap L= \\
             & =[a]\big(\Theta(E)\big)\cap L=[a]E
\end{align*}
and, analogously, $[d]E\vee[a]E=[a]E$. This shows $[a]E\in U([c]E,[d]E)$, a contradiction. Hence $[a]\big(\Theta(E)\big)\neq[1_{\mathbf L^*}]\big(\Theta(E)\big)$. Analogously, one obtains that in case $0_{(\mathbf L/E)^*}\in(L/E)^*$ we have $[a]\big(\Theta(E)\big)\neq[0_{\mathbf L^*}]\big(\Theta(E)\big)$. Now assume $0_{(\mathbf L/E)^*},1_{(\mathbf L/E)^*}\in(L/E)^*$ and $[0_{\mathbf L^*}]\big(\Theta(E)\big)=[1_{\mathbf L^*}]\big(\Theta(E)\big)$. Then $0_{\mathbf L^*},1_{\mathbf L^*}\in L^*$ and $(0_{\mathbf L^*},1_{\mathbf L^*})\in\Theta(E)$ and hence $\Theta(E)=(L^*)^2$ whence
\[
E=\Theta(E)\cap L^2=(L^*)^2\cap L^2=L^2
\]
which implies $|L/E|=1$ and hence $0_{(\mathbf L/E)^*},1_{(\mathbf L/E)^*}\notin(L/E)^*$, a contradiction. \\
$\bullet$ $f$ is surjective. \\
Assume $1_{\mathbf L^*}\in L^*$. If $1_{(\mathbf L/E)^*}\in(L/E)^*$ then $f(1_{(\mathbf L/E)^*})=[1_{\mathbf L^*}]\big(\Theta(E)\big)$. Now assume $1_{(\mathbf L/E)^*}\notin(L/E)^*$. Then $[x]E\vee[y]E$ is defined in $L/E$ for every $x,y\in L$. Suppose $[1_{\mathbf L^*}]\big(\Theta(E)\big)=\{1_{\mathbf L^*}\}$. Then $U(x,y)\neq\emptyset$ for all $x,y\in L$ and hence $1_{\mathbf L^*}\notin L^*$, a contradiction. Therefore $[1_{\mathbf L^*}]\big(\Theta(E)\big)\neq\{1_{\mathbf L^*}\}$. According to the proof of Lemma~\ref{lem1} there exists some $\alpha\in[1_{\mathbf L^*}]\big(\Theta(E)\big)\cap L$ and we obtain $f([\alpha]E)=[\alpha]\big(\Theta(E)\big)=[1_{\mathbf L^*}]\big(\Theta(E)\big)$. Analogously, one can show that in case $0_{\mathbf L^*}\in L^*$ there exists some $x\in(L/E)^*$ with $f(x)=[0_{\mathbf L^*}]\big(\Theta(E)\big)$. \\
$\bullet$ $f$ is a homomorphism from $(\mathbf L/E)^*$ to $\mathbf L^*/\big(\Theta(E)\big)$. \\
First assume $U(a,b)\neq\emptyset$. Then
\begin{align*}
f([a]E\vee[b]E) & =f\Big([a\vee^*b]\big(\Theta(E)\big)\cap L\Big)=f\Big([a\vee b]\big(\Theta(E)\big)\cap L\Big)=f([a\vee b]E)= \\
                & =[a\vee b]\big(\Theta(E)\big)=[a]\big(\Theta(E)\big)\vee[b]\big(\Theta(E)\big)=f([a]E)\vee f([b]E).
\end{align*}
Now assume $U(a,b)=\emptyset$. Then $1_{\mathbf L^*}\in L^*$. First assume $[1_{\mathbf L^*}]\big(\Theta(E)\big)=\{1_{\mathbf L^*}\}$. Then $1_{(\mathbf L/E)^*}\in(L/E)^*$ and
\begin{align*}
f([a]E\vee[b]E) & =f(1_{(\mathbf L/E)^*})=[1_{\mathbf L^*}]\big(\Theta(E)\big)=[a\vee b]\big(\Theta(E)\big)=[a]\big(\Theta(E)\big)\vee[b]\big(\Theta(E)\big)= \\
                & =f([a]E)\vee f([b]E).
\end{align*}
Finally, assume $[1_{\mathbf L^*}]\big(\Theta(E)\big)\neq\{1_{\mathbf L^*}\}$. According to the proof of Lemma~\ref{lem1} there exists some $\alpha\in[1_{\mathbf L^*}]\big(\Theta(E)\big)\cap L$ and we have
\begin{align*}
f([a]E\vee[b]E) & =f\Big([a\vee^*b]\big(\Theta(E)\big)\cap L\Big)=f\Big([1_{\mathbf L^*}]\big(\Theta(E)\big)\cap L\Big)=f\Big([\alpha]\big(\Theta(E)\big)\cap L\Big)= \\
                & =f([\alpha]E)=[\alpha]\big(\Theta(E)\big)=[1_{\mathbf L^*}]\big(\Theta(E)\big)=[a\vee b]\big(\Theta(E)\big)= \\
                & =[a]\big(\Theta(E)\big)\vee[b]\big(\Theta(E)\big)=f([a]E)\vee f([b]E).
\end{align*}
It is easy to see that $f(x\vee y)=f(x)\vee f(y)$ for all $(x,y)\in\big((L/E)^*\big)^2\setminus(L/E)^2$. Analogously, one can prove $f(x\wedge y)=f(x)\wedge f(y)$ for all $x,y\in(L/E)^*$.
\end{proof}

\section{Examples}

In the following we present several examples showing different situations concerning two-point extensions and quotient partial lattices.

\begin{example}
Let $\mathbf L$ denote the partial lattice visualized in Figure~4:

\vspace*{-3mm}

\begin{center}
\setlength{\unitlength}{7mm}
\begin{picture}(3,4)
\put(0,1){\circle*{.3}}
\put(0,3){\circle*{.3}}
\put(2,2){\circle*{.3}}
\put(0,1){\line(0,1)2}
\put(-.15,.25){$a$}
\put(-.15,3.4){$c$}
\put(2.4,1.85){$b$}
\put(.2,-.8){{\rm Fig.~4}}
\end{picture}
\end{center}

\vspace*{3mm}

The lattice $\mathbf L^*$ is depicted in Figure~5:

\vspace*{-3mm}

\begin{center}
\setlength{\unitlength}{7mm}
\begin{picture}(6,8)
\put(3,1){\circle*{.3}}
\put(1,3){\circle*{.3}}
\put(5,4){\circle*{.3}}
\put(1,5){\circle*{.3}}
\put(3,7){\circle*{.3}}
\put(3,1){\line(-1,1)2}
\put(3,1){\line(2,3)2}
\put(3,7){\line(-1,-1)2}
\put(3,7){\line(2,-3)2}
\put(1,3){\line(0,1)2}
\put(2.6,.25){$0_{\mathbf L^*}$}
\put(.3,2.85){$a$}
\put(.3,4.85){$c$}
\put(5.4,3.85){$b$}
\put(2.6,7.4){$1_{\mathbf L^*}$}
\put(2.2,-.75){{\rm Fig.~5}}
\end{picture}
\end{center}

\vspace*{3mm}

Here $0_{\mathbf L^*},1_{\mathbf L^*}\in L^*$. Put $E:=\{a,c\}^2\cup\{b\}^2$ Then $\Theta(E)=\{0_{\mathbf L^*}\}^2\cup\{a,c\}^2\cup\{b\}^2\cup\{1_{\mathbf L^*}\}^2$ and hence $\Theta(E)$ is a congruence on $\mathbf L^*$ satisfying $\Theta(E)\cap L^2=E$ and $\mathbf L/E$ is visualized in Figure~6:
\begin{center}
\setlength{\unitlength}{7mm}
\begin{picture}(2,1)
\put(0,1){\circle*{.3}}
\put(2,1){\circle*{.3}}
\put(-.6,.25){$[a]E$}
\put(1.4,.25){$[b]E$}
\put(.2,-1.1){{\rm Fig.~6}}
\end{picture}
\end{center}

\vspace*{5mm}

The lattice $(\mathbf L/E)^*$ is depicted in Figure~7:

\vspace*{-3mm}

\begin{center}
\setlength{\unitlength}{7mm}
\begin{picture}(6,6)
\put(3,1){\circle*{.3}}
\put(1,3){\circle*{.3}}
\put(5,3){\circle*{.3}}
\put(3,5){\circle*{.3}}
\put(3,1){\line(-1,1)2}
\put(3,1){\line(1,1)2}
\put(3,5){\line(-1,-1)2}
\put(3,5){\line(1,-1)2}
\put(2.2,.25){$0_{(\mathbf L/E)^*}$}
\put(-.4,2.85){$[a]E$}
\put(5.3,2.85){$[b]E$}
\put(2.2,5.4){$1_{(\mathbf L/E)^*}$}
\put(2.2,-.85){{\rm Fig.~7}}
\end{picture}
\end{center}

\vspace*{3mm}

and we have $0_{(\mathbf L/E)^*},1_{(\mathbf L/E)^*}\in(L/E)^*$. Finally, $\mathbf L^*/\big(\Theta(E)\big)$ is visualized in Figure~8:

\vspace*{-1mm}

\begin{center}
\setlength{\unitlength}{7mm}
\begin{picture}(6,6)
\put(3,1){\circle*{.3}}
\put(1,3){\circle*{.3}}
\put(5,3){\circle*{.3}}
\put(3,5){\circle*{.3}}
\put(3,1){\line(-1,1)2}
\put(3,1){\line(1,1)2}
\put(3,5){\line(-1,-1)2}
\put(3,5){\line(1,-1)2}
\put(1.5,.25){$[0_{\mathbf L^*}]\big(\Theta(E)\big)$}
\put(-1.8,2.85){$[a]\big(\Theta(E)\big)$}
\put(5.3,2.85){$[b]\big(\Theta(E)\big)$}
\put(1.5,5.4){$[1_{\mathbf L^*}]\big(\Theta(E)\big)$}
\put(2.2,-.85){{\rm Fig.~8}}
\end{picture}
\end{center}

\vspace*{3mm}

In accordance with Theorem~\ref{th1}, $(\mathbf L/E)^*\cong\mathbf L^*/\big(\Theta(E)\big)$.

\end{example}

\begin{example}\label{ex1}
Let $\mathbf L$ be the partial lattice depicted in Figure~9:

\vspace*{-3mm}

\begin{center}
\setlength{\unitlength}{7mm}
\begin{picture}(6,4)
\put(0,1){\circle*{.3}}
\put(2,3){\circle*{.3}}
\put(4,1){\circle*{.3}}
\put(6,3){\circle*{.3}}
\put(2,3){\line(-1,-1)2}
\put(2,3){\line(1,-1)2}
\put(4,1){\line(1,1)2}
\put(-.15,.25){$a$}
\put(3.85,.25){$b$}
\put(1.85,3.4){$c$}
\put(5.85,3.4){$d$}
\put(2.2,-.75){{\rm Fig.~9}}
\end{picture}
\end{center}

\vspace*{3mm}

The lattice $\mathbf L^*$ is visualized in Figure~10:

\vspace*{-3mm}

\begin{center}
\setlength{\unitlength}{7mm}
\begin{picture}(8,8)
\put(3,1){\circle*{.3}}
\put(1,3){\circle*{.3}}
\put(5,3){\circle*{.3}}
\put(3,5){\circle*{.3}}
\put(7,5){\circle*{.3}}
\put(5,7){\circle*{.3}}
\put(3,1){\line(-1,1)2}
\put(3,1){\line(1,1)4}
\put(5,3){\line(-1,1)2}
\put(5,7){\line(-1,-1)4}
\put(5,7){\line(1,-1)2}
\put(2.6,.25){$0_{\mathbf L^*}$}
\put(.3,2.85){$a$}
\put(5.4,2.85){$b$}
\put(2.3,4.85){$c$}
\put(7.4,4.85){$d$}
\put(4.6,7.4){$1_{\mathbf L^*}$}
\put(3.1,-.75){{\rm Fig.~10}}
\end{picture}
\end{center}

\vspace*{3mm}

Here $0_{\mathbf L^*},1_{\mathbf L^*}\in L^*$. Put $E:=\{a\}^2\cup\{b,d\}^2\cup\{c\}^2$. Let us mention that $(b,d)\in E$ and $a\vee b=c$, but $a\vee d$ does not exist in $\mathbf L$. We have $\Theta(E)=\{0_{\mathbf L^*}\}^2\cup\{a\}^2\cup\{b,d\}^2\cup\{c,1_{\mathbf L^*}\}^2$ and hence $\Theta(E)\cap L^2=E$ and $\mathbf L/E$ is depicted in Figure~11:

\vspace*{-3mm}

\begin{center}
\setlength{\unitlength}{7mm}
\begin{picture}(4,4)
\put(0,1){\circle*{.3}}
\put(4,1){\circle*{.3}}
\put(2,3){\circle*{.3}}
\put(2,3){\line(-1,-1)2}
\put(2,3){\line(1,-1)2}
\put(-.6,.25){$[a]E$}
\put(3.4,.25){$[d]E$}
\put(1.4,3.4){$[c]E$}
\put(1.1,-.85){{\rm Fig.~11}}
\end{picture}
\end{center}

\vspace*{3mm}

One can see that although $[a]E\vee[d]E$ exists in $\mathbf L/E$, $a\vee d$ does not exist in $\mathbf L$. The lattice $(\mathbf L/E)^*$ is visualized in Figure~12:

\vspace*{-3mm}

\begin{center}
\setlength{\unitlength}{7mm}
\begin{picture}(6,6)
\put(3,1){\circle*{.3}}
\put(1,3){\circle*{.3}}
\put(5,3){\circle*{.3}}
\put(3,5){\circle*{.3}}
\put(3,1){\line(-1,1)2}
\put(3,1){\line(1,1)2}
\put(3,5){\line(-1,-1)2}
\put(3,5){\line(1,-1)2}
\put(2.2,.25){$0_{(\mathbf L/E)^*}$}
\put(-.4,2.85){$[a]E$}
\put(5.3,2.85){$[d]E$}
\put(2.4,5.4){$[c]E$}
\put(2.1,-.85){{\rm Fig.~12}}
\end{picture}
\end{center}

\vspace*{3mm}

and we have $0_{(\mathbf L/E)^*}\in(L/E)^*$ and $1_{(\mathbf L/E)^*}\notin(L/E)^*$. Finally, $\mathbf L^*/\big(\Theta(E)\big)$ is depicted in Figure~13:

\vspace*{-3mm}

\begin{center}
\setlength{\unitlength}{7mm}
\begin{picture}(6,6)
\put(3,1){\circle*{.3}}
\put(1,3){\circle*{.3}}
\put(5,3){\circle*{.3}}
\put(3,5){\circle*{.3}}
\put(3,1){\line(-1,1)2}
\put(3,1){\line(1,1)2}
\put(3,5){\line(-1,-1)2}
\put(3,5){\line(1,-1)2}
\put(1.5,.25){$[0_{\mathbf L^*}]\big(\Theta(E)\big)$}
\put(-1.8,2.85){$[a]\big(\Theta(E)\big)$}
\put(5.3,2.85){$[b]\big(\Theta(E)\big)$}
\put(1.8,5.4){$[c]\big(\Theta(E)\big)$}
\put(2.1,-.85){{\rm Fig.~13}}
\end{picture}
\end{center}

\vspace*{3mm}

\end{example}

\begin{example}
Let $\mathbf L$ denote the partial lattice from Example~\ref{ex1}. Then $0_{\mathbf L^*},1_{\mathbf L^*}\in L^*$. Put $E:=\{a,c\}^2\cup\{b\}^2\cup\{d\}^2$. Here $(a,c)\in E$ and $c\wedge d=b$, but $a\wedge d$ does not exist in $\mathbf L$. We have $\Theta(E)=\{0_{\mathbf L^*},b\}^2\cup\{a,c\}^2\cup\{d\}^2\cup\{1_{\mathbf L^*}\}^2$ and hence $\Theta(E)\cap L^2=E$ and $\mathbf L/E$ is visualized in Figure~14:

\vspace*{-3mm}

\begin{center}
\setlength{\unitlength}{7mm}
\begin{picture}(4,4)
\put(2,1){\circle*{.3}}
\put(0,3){\circle*{.3}}
\put(4,3){\circle*{.3}}
\put(2,1){\line(-1,1)2}
\put(2,1){\line(1,1)2}
\put(1.4,.25){$[b]E$}
\put(-.6,3.4){$[a]E$}
\put(3.4,3.4){$[d]E$}
\put(1.1,-.85){{\rm Fig.~14}}
\end{picture}
\end{center}

\vspace*{3mm}

One can see that although $[a]E\wedge[b]E$ exists in $\mathbf L/E$, $a\wedge b$ does not exists in $\mathbf L$. The lattice $(\mathbf L/E)^*$ is depicted in Figure~15:

\vspace*{-3mm}

\begin{center}
\setlength{\unitlength}{7mm}
\begin{picture}(6,6)
\put(3,1){\circle*{.3}}
\put(1,3){\circle*{.3}}
\put(5,3){\circle*{.3}}
\put(3,5){\circle*{.3}}
\put(3,1){\line(-1,1)2}
\put(3,1){\line(1,1)2}
\put(3,5){\line(-1,-1)2}
\put(3,5){\line(1,-1)2}
\put(2.4,.25){$[b]E$}
\put(-.4,2.85){$[a]E$}
\put(5.3,2.85){$[d]E$}
\put(1.8,5.4){$[1_{(\mathbf L/E)^*}]E$}
\put(2.1,-.85){{\rm Fig.~15}}
\end{picture}
\end{center}

\vspace*{3mm}

and we have $0_{(\mathbf L/E)^*}\notin(L/E)^*$ and $1_{(\mathbf L/E)^*}\in(L/E)^*$. Finally, $\mathbf L^*/\big(\Theta(E)\big)$ is visualized in Figure~16:

\vspace*{-3mm}

\begin{center}
\setlength{\unitlength}{7mm}
\begin{picture}(6,6)
\put(3,1){\circle*{.3}}
\put(1,3){\circle*{.3}}
\put(5,3){\circle*{.3}}
\put(3,5){\circle*{.3}}
\put(3,1){\line(-1,1)2}
\put(3,1){\line(1,1)2}
\put(3,5){\line(-1,-1)2}
\put(3,5){\line(1,-1)2}
\put(1.5,.25){$[0_{\mathbf L^*}]\big(\Theta(E)\big)$}
\put(-1.8,2.85){$[a]\big(\Theta(E)\big)$}
\put(5.3,2.85){$[d]\big(\Theta(E)\big)E$}
\put(1.5,5.4){$[1_{\mathbf L^*}]\big(\Theta(E)\big)$}
\put(2.1,-.85){{\rm Fig.~16}}
\end{picture}
\end{center}

\vspace*{3mm}

\end{example}

The next example shows that $\mathbf L/E$ may be a lattice even if $\mathbf L$ is only a partial lattice.

\begin{example}
Let $\mathbf L$ denote the partial lattice from Example~\ref{ex1}. Then $0_{\mathbf L^*},1_{\mathbf L^*}\in L^*$. Put $E:=\{a\}^2\cup\{b,c\}^2\cup\{d\}^2$. Then $\Theta(E)=\{0_{\mathbf L^*},a\}^2\cup\{b,c\}^2\cup\{d,1_{\mathbf L^*}\}^2$ and hence $\Theta(E)\cap L^2=E$ and $\mathbf L/E$ is depicted in Figure~17:

\begin{center}
\setlength{\unitlength}{7mm}
\begin{picture}(1,4)
\put(0,0){\circle*{.3}}
\put(0,2){\circle*{.3}}
\put(0,4){\circle*{.3}}
\put(0,0){\line(0,1)4}
\put(.4,-.15){$[a]E$}
\put(.4,1.85){$[b]E$}
\put(.4,3.85){$[d]E$}
\put(-.9,-1.3){{\rm Fig.~17}}
\end{picture}
\end{center}

\vspace*{7mm}

Hence $(\mathbf L/E)^*=\mathbf L/E$ and $0_{(\mathbf L/E)^*},1_{(\mathbf L/E)^*}\notin(L/E)^*$. Finally, $\mathbf L^*/\big(\Theta(E)\big)$ is visualized in Figure~18:

\begin{center}
\setlength{\unitlength}{7mm}
\begin{picture}(1,4)
\put(0,0){\circle*{.3}}
\put(0,2){\circle*{.3}}
\put(0,4){\circle*{.3}}
\put(0,0){\line(0,1)4}
\put(.4,-.15){$[0_{\mathbf L^*}]\big(\Theta(E)\big)$}
\put(.4,1.85){$[b]\big(\Theta(E)\big)$}
\put(.4,3.85){$[d]\big(\Theta(E)\big)$}
\put(-.9,-1.45){{\rm Fig.~18}}
\end{picture}
\end{center}

\vspace*{7mm}

\end{example}

\section{Conclusion}

We introduced the so-called two-point extension of a partial lattice which extends it to a lattice with everywhere defined operations. This was not possible by using the one-point extension which was intensively used for partial algebras in general because the one-point extension of a partial lattice need not be a lattice. However, this is not a final step concerning this research. Namely lattice distributivity or modularity can be defined for partial lattices by strong and regular identities but these are not preserved by the two-point extension. For example, if a partial lattice $\mathbf L=(L,\vee,\wedge)$ is a finite antichain containing $n\geq3$ elements then the join, respectively meet of two distinct elements of $L$ is not defined and hence $\mathbf L$ trivially satisfies the strong distributive identity. However, its two-point extension $\mathbf L^*$ is isomorphic to the non-modular lattice $\mathbf M_n$. Hence, this research may continue with finding other tools which avoid this difficulty.

Authors' addresses:

Ivan Chajda \\
Palack\'y University Olomouc \\
Faculty of Science \\
Department of Algebra and Geometry \\
17.\ listopadu 12 \\
771 46 Olomouc \\
Czech Republic \\
ivan.chajda@upol.cz

Helmut L\"anger \\
TU Wien \\
Faculty of Mathematics and Geoinformation \\
Institute of Discrete Mathematics and Geometry \\
Wiedner Hauptstra\ss e 8-10 \\
1040 Vienna \\
Austria, and \\
Palack\'y University Olomouc \\
Faculty of Science \\
Department of Algebra and Geometry \\
17.\ listopadu 12 \\
771 46 Olomouc \\
Czech Republic \\
helmut.laenger@tuwien.ac.at

\begin{thebibliography}9
\bibitem{B86}
P.~Burmeister, A Model Theoretic Oriented Approach to Partial Algebras. Math.\ Research {\bf32}, Akademie-Verlag, Berlin 1986.
\bibitem{B93}
P.~Burmeister, Partial algebras -- an introductory survey. Algebras and Orders (Montreal, 1991), 1--70, NATO Adv.\ Sci.\ Inst.\ Ser.\ C: Math.\ Phys.\ Sci.\ {\bf389}, Kluwer, Dordrecht 1993.
\bibitem{CS}
I.~Chajda and Z.~Seidl, An algebraic approach to partial lattices. Demonstratio Math.\ {\bf30} (1997), 485--494.
\bibitem D
K.~Denecke, Strong regular varieties of partial algebras. I. Beitr\"age Algebra Geom.\ {\bf31} (1991), 163--177.
\bibitem{SS}
B.~Staruch and B.~Staruch, Strong regular varieties of partial algebras. Algebra Universalis {\bf31} (1994), 157--176.
\end{thebibliography}
\end{document}